\documentclass[10pt,leqno]{amsart}
\usepackage{amsmath}
\usepackage{amsfonts}
\usepackage{amssymb}
\usepackage{graphicx}
\usepackage{color}
\usepackage{hyperref}
\usepackage{xcolor}
\newtheorem{theorem}{Theorem}[section]
\newtheorem*{theorem*}{Theorem}
\newtheorem{lemma}{Lemma}[section]
\theoremstyle{definition}
\newtheorem{definition}[theorem]{Definition}

\newtheorem{remark}[theorem]{Remark}

\newcommand{\R}{\mathbb{R}}

\def \b {\beta}

\def\Ric{\text{Ric}}

\def\a{\alpha}
\def\l{\lambda}

\def\R{\mathbb{R}}
\def\Z{\mathbb{Z}}

\def\vp{\varphi}

\def\id{\operatorname{id}}
\def\Ric{\operatorname{Ric}}

\def\tr{\operatorname{tr}}

\numberwithin{equation}{section}
\makeatletter
\newcommand*\owedge{\mathpalette\@owedge\relax}
\newcommand*\@owedge[1]{%
  \mathbin{%
    \ooalign{%
      $#1\m@th\bigcirc$\cr
      \hidewidth$#1\m@th\wedge$\hidewidth\cr
    }%
  }%
}
\makeatother

\begin{document}

\title[K\"ahler surfaces and curvature operator of the second kind]{K\"ahler surfaces with six-positive curvature operator of the second kind}

\author{Xiaolong Li}\thanks{The author's research is partially supported by Simons Collaboration Grant \#962228 and a start-up grant at Wichita State University}
\address{Department of Mathematics, Statistics and Physics, Wichita State University, Wichita, KS, 67260}
\email{xiaolong.li@wichita.edu}


\begin{abstract}
The purpose of this article is to initiate the investigation of the curvature operator of the second kind on K\"ahler manifolds. 
The main result asserts that a closed K\"ahler surface with six-positive curvature operator of the second kind is biholomorphic to $\mathbb{CP}^2$.
It is also shown that a closed non-flat K\"ahler surface with six-nonnegative curvature operator of the second kind is either biholomorphic to $\mathbb{CP}^2$ or isometric to $\mathbb{S}^2 \times \mathbb{S}^2$. 
\end{abstract}

\maketitle

\section{Introduction}
The Riemann curvature tensor $R_{ijkl}$ of a Riemannian manifold $(M^n,g)$ defines two kinds of curvature operators:  $\hat{R}$ acting on two-forms via 
\begin{equation*}
    \hat{R}(e_i\wedge e_j) =\frac 1 2 \sum_{k,l}R_{ijkl}e_k \wedge e_l,
\end{equation*}
and $\mathring{R}$ acting on symmetric two-tensors
via 
\begin{equation*}
    \mathring{R}(e_i \odot e_j) =\sum_{k,l}R_{iklj} e_k \odot e_l,
\end{equation*}
where $\{e_1, \cdots, e_n\}$ is an orthonormal basis of the tangent space $T_pM$ at $p \in M$, $\wedge$ denotes the wedge product, and $\odot$ denotes the symmetric product. 
The self-adjoint operator $\hat{R}$ is the so-called curvature operator (of the first kind by Nishikawa \cite{Nishikawa86}) and it is well-studied in Riemannian geometry (see \cite{GM75}, \cite{Tachibana74}, \cite{Hamilton86}, \cite{BW08}, \cite{PW21}, etc.). The curvature operator of the second kind refers to the restriction of $\mathring{R}$ to the space of traceless symmetric two-tensors $S^2_0(T_pM)$ (see Section 2 for more details). 
The operator $\mathring{R}$ and its restriction to $S^2_0(T_pM)$  were investigated in some early works (see \cite{CV60}, \cite{BE69}, \cite{BK78}, \cite{OT79}, \cite{Nishikawa86}, \cite{Kashiwada93} etc.), and some recently works resolving Nishikawa's conjecture (see \cite{CGT21} and \cite{Li21, Li22JGA}). 

To motivate the context of this paper, let's begin with an overview on Nishikawa's conjecture and its resolution. 
A central theme in geometry is that positivity conditions on curvature often have strong restrictions on the topology of the underlying manifold.
For instance, a fundamental result in Riemannian geometry asserts that the $n$-sphere $\mathbb{S}^n$, up to diffeomorphism, is the only simply-connected closed manifold that admits a metric with two-positive curvature operator. This was proved via Ricci flow by Hamilton \cite{Hamilton82} for $n=3$, Hamilton \cite{Hamilton86} and Chen \cite{Chen91} for $n=4$, and B\"ohm and Wilking \cite{BW08} for $n\geq 5$.
Regarding the curvature operator of the second kind, Nishikawa \cite{Nishikawa86} conjectured in 1986 that a closed Riemannian manifold with positive curvature operator of the second kind is diffeomorphic to a spherical space form and a closed Riemannian manifold with nonnegative curvature operator of the second kind is diffeomorphic to a Riemannian locally symmetric space. The positive and nonnegative case of Nishikawa's conjecture were recently settled under weaker assumptions by Cao, Gursky and Tran \cite{CGT21} and the author \cite{Li21}, respectively. More precisely, it has been shown that
\begin{theorem}\label{thm Nishikawa}
Let $(M^n,g)$ be a closed Riemannian manifold of dimension $n\geq 3$. 
\begin{enumerate}
    \item If $M$ have three-positive curvature operator of the second kind, then $M$ is diffeomorphic to a spherical space form. 
    \item If $M$ have three-nonnegative curvature operator of the second kind, then $M$ is either flat, or diffeomorphic to a spherical space form, or $n\geq 5$ and $M$ is isometric to a compact irreducible symmetric space\footnote{After the submission of this paper, the third possibility in part (2) has been ruled out by Nienhaus, Petersen and Wink \cite{NPW22}. They proved that a closed $n$-dimensional Riemannian manifold with $\frac{n+2}{2}$-nonnegative curvature operator of the second kind is either a rational homology sphere or flat.}. 
\end{enumerate}
\end{theorem}

Part (1) of Theorem \ref{thm Nishikawa} was proved by Cao, Gursky and Tran \cite{CGT21} under two-positivity of the curvature operator of the second kind. Their key observation is that two-positive curvature operator of the second kind implies the strictly PIC1 condition introduced by Brendle \cite{Brendle08}. This is sufficient, as Brendle \cite{Brendle08} has shown that the normalized Ricci flow evolves an initial metric with strictly PIC1 into a limit metric with constant positive sectional curvature.
Soon after, the author \cite{Li21} weakened the assumption to three-positivity by showing that strictly PIC1 is implied by three-positivity of the curvature operator of the second kind. Furthermore, the author proved part (2) of Theorem \ref{thm Nishikawa} by reducing the problem to the locally irreducible case and using the classification results (see \cite{Brendle10book}) for closed locally irreducible Riemannian manifolds with weakly PIC1.

Another important result obtained by Cao, Gursky and Tran in \cite{CGT21} states that when $n\geq 4$, $\mathbb{S}^n$ is, up to homeomorphism, the only simply-connected closed manifold which admits a metric with four-positive curvature operator of the second kind. 
This follows from their important observation that four-positive curvature operator of the second kind implies positive isotropic curvature, and the excellent work of Micallef and Moore \cite{MM88} on closed simply-connected manifolds with positive isotropic curvature. An improvement of this result, together with a corresponding rigidity theorem, has been obtained by the author in \cite{Li22JGA}.
\begin{theorem}\label{thm 4.5PCO}
Let $(M^n,g)$ be a closed Riemannian manifold of dimension $n\geq 4$. If $M$ has $4\frac1 2$-positive curvature operator of the second kind, then $M$ is homeomorphic (diffeomorphic if either $n=4$ or $n\geq 12$) to a spherical space form. If $M$ has $4\frac 1 2$-nonnegative curvature operator of the second kind, then one of the following statements holds:
\begin{enumerate}
    \item $M$ is flat;
    \item $M$ is homeomorphic (diffeomorphic if either $n=4$ or $n\geq 12$) to a spherical space form;
    \item $n=4$ and $M$ is isometric to $\mathbb{CP}^2$ with Fubini-Study metric, up to scaling\footnote{The isometry is obtained in Theorem \ref{thm Kahler 2D} and also in \cite{Li22Kahler}.};
    \item $n=4$ and the universal cover of $M$ is isometric to $\mathbb{S}^3 \times \R$, up to scaling\footnote{The isometry is obtained in \cite{Li22product}.}.
\end{enumerate}
\end{theorem}

In the above theorem, $(M^n,g)$ is said to have $\alpha$-positive curvature operator of the second kind for $\a \in [1,\frac{(n-1)(n+2)}{2}]$ if for any $p\in M$ and any orthonormal basis $\{\vp_i\}_{i=1}^{\frac{(n-1)(n+2)}{2}}$ of $S^2_0(T_pM)$, it holds that
\begin{equation}\label{alpha positive def}
     \sum_{i=1}^{\lfloor \a \rfloor} \mathring{R}(\vp_i,\vp_i) +(\a -\lfloor \a \rfloor) \mathring{R}(\vp_{\lfloor \a \rfloor+1},\vp_{\lfloor \a \rfloor+1}) > 0. 
\end{equation}
Here and throughout the paper, $\lfloor \a \rfloor $ denotes the floor function defined by 
$$\lfloor \a \rfloor =\max \{ m \in \Z: m \leq \a \}.$$
When $\a=k$ is a positive integer, this reduces to the usual definition, which means the sum of the smallest $k$ eigenvalues of the curvature operator of the second kind is positive. Of course $\alpha$-nonnegativity means equality is allowed in \eqref{alpha positive def}.

The purpose of this paper is to initiate the study of the curvature operator of the second kind on K\"ahler manifolds. A natural question is to seek for a positivity condition on the curvature operator of the second kind that characterizes the complex projective space $\mathbb{CP}^m$. We are able to answer this question in complex dimension two. 

\begin{theorem}\label{thm main}
Let $(M^4,g)$ be a closed K\"ahler surface with six-positive curvature operator of the second kind. Then $M$ is biholomorphic to $\mathbb{CP}^2$. 
\end{theorem}

The curvature assumption in Theorem \ref{thm main} is optimal, as the product K\"ahler manifold $\mathbb{S}^2 \times \mathbb{S}^2$, with the same round metric on both factors, has $(6+\epsilon)$-positive curvature operator of the second kind for any $\epsilon >0$ (see \cite{BK78} or \cite{CGT21}).  

We also prove a rigidity result. 
\begin{theorem}\label{thm 6NCO}
Let $(M^4,g)$ be a closed K\"ahler surface with six-nonnegative curvature operator of the second kind. Then one of the following statements holds: 
\begin{enumerate}
    \item $M$ is flat;
    \item $M$ is biholomorphic to $\mathbb{CP}^2$;
    \item the universal cover of $M$ is isometric to $\mathbb{S}^2\times \mathbb{S}^2$, with both factors equipped with the same round metric. 
\end{enumerate}
\end{theorem}

In K\"ahler geometry, the complex projective space $\mathbb{CP}^m$ is characterized by several positivity conditions on curvature. 
For instance, $\mathbb{CP}^m$ is, up to biholomorphism, the only closed manifold that admits a K\"ahler metric with positive bisectional curvature.  
This was known as the Frankel conjecture \cite{Frankel61} and it was proved independently by Mori \cite{Mori79} and Siu and Yau \cite{SY80}. 
A weaker condition, called positive orthogonal bisectional curvature, also characterizes $\mathbb{CP}^m$. This is due to Chen \cite{Chen07} and Gu and Zhang \cite{GZ10} (see also Wilking \cite{Wilking13} for an alternative proof using Ricci flow).
Therefore, Theorem \ref{thm main} is a direct consequence of the following theorem.  
\begin{theorem}\label{thm 6PCO kahler}
Let $(M^m,g)$ be a K\"ahler manifold of complex dimension $m\geq 2$. If $M$ has six-positive (respectively, six-nonnegative) curvature operator of the second kind, then $M$ has positive (respectively, nonnegative) orthogonal bisectional curvature. 
\end{theorem}

Naturally, one might wonder what can be said for a Riemannian manifold with six-positive curvature operator of the second kind. Regarding this question, we prove that
\begin{theorem}\label{thm 6PCO Riem}
Let $(M^n,g)$ be a Riemannian manifold of dimension $n\geq 4$. If $M$ has six-positive (respectively, six-nonnegative) curvature operator of the second kind, then for any $p\in M$ and any orthonormal four-frame $\{e_1, e_2, e_3, e_4\}$ in $T_pM$, it holds that 
\begin{equation}\label{eq 4 sec}
    R_{1313}+R_{1414}+R_{2323}+R_{2424} > (\text{respectively, $\geq$}) \  0. 
\end{equation}
\end{theorem}

Note that Theorem \ref{thm 6PCO Riem} is sharp on $\mathbb{S}^2 \times \mathbb{S}^2$, whose curvature operator of the second kind is six-nonnegative (but not six-positive) and \eqref{eq 4 sec} holds as an equality for suitable four-frames. 
Indeed, the strategy to prove Theorem \ref{thm 6PCO Riem} is to use $\mathbb{S}^2 \times \mathbb{S}^2$ as a model space.
It is worth mentioning that a similar strategy has been used to prove Theorem \ref{thm 4.5PCO} in \cite{Li22JGA} with $\mathbb{CP}^2$ and $\mathbb{S}^3 \times \mathbb{S}^1$ as model spaces. 
Although \eqref{eq 4 sec} alone does not even imply positive/nonnegative Ricci curvature for Riemannian manifolds, it implies positive/nonnegative orthogonal bisectional curvature for K\"ahler manifolds.

In addition, we use the normal form of $\mathring{R}$ constructed by Cao, Gursky and Tran in \cite{CGT21} for oriented four-manifolds to prove a rigidity result for $\mathbb{CP}^2$. 
\begin{theorem}\label{thm Kahler 2D}
A closed non-flat K\"ahler surface with $4\frac{1}{2}$-nonnegative curvature operator is isometric to $\mathbb{CP}^2$ with the Fubini-Study metric.
\end{theorem}

Finally, we would like to point out that some results in this paper, including Theorems \ref{thm main}, \ref{thm 6PCO kahler} and \ref{thm Kahler 2D}, have been generalized to higher dimensions in a later work by the author \cite{Li22Kahler}.

This paper is organized as follows. 
In Section 2, we give an introduction to the curvature operator of the second kind.
The proof of Theorem \ref{thm 6PCO Riem} is presented in Section 3. 
In Section 4, we prove Theorems \ref{thm 6PCO kahler} and \ref{thm main}. 
The proofs of Theorems \ref{thm 6NCO} and \ref{thm Kahler 2D} are given in Sections 5 and 6, respectively.

\section{The curvature operator of the second kind}
Let $(V,g)$ be a Euclidean vector space of dimension $n\geq 2$.  We always identify $V$ with its dual space $V^*$ via the metric $g$. 
The space of bilinear forms on $V$ is denote by $T^2(V)$, and it splits as 
\begin{equation*}
    T^2(V)=S^2(V)\oplus  \Lambda^2(V),
\end{equation*}
where $S^2(V)$ is the space of symmetric two-tensors on $V$ and $\Lambda^2(V)$ is the space of two-forms on $V$.
Our conventions on symmetric products and wedge products are that, for $u$ and $v$ in $V$, $\odot$ denotes the symmetric product defined by 
\begin{equation*}
    u \odot v  =u \otimes v + v \otimes u,
\end{equation*}
and 
$\wedge$ denotes the wedge product defined by 
\begin{equation*}
    u \wedge v  =u \otimes v -v \otimes u.
\end{equation*}
The inner product $g$ on $V$ naturally induces inner products on $S^2(V)$ and $\Lambda^2(V)$, respectively. 
To be consistent with \cite{CGT21}, the inner product on $S^2(V)$ is defined as 
\begin{equation*}
    \langle A, B \rangle =\tr(A^T B),
\end{equation*}
and the inner product on $\Lambda^2(V)$ is defined as 
\begin{equation*}
    \langle A, B \rangle =\frac{1}{2}\tr(A^T B).
\end{equation*}
In particular, if $\{e_1, \cdots, e_n\}$ is an orthonormal basis for $V$, then 
$\{e_i \wedge e_j\}_{1\leq i < j \leq n}$ is an orthonormal basis for $\Lambda^2(V)$ and $\{\frac{1}{\sqrt{2}}e_i \odot e_j\}_{1\leq i < j \leq n} \cup \{\frac{1}{2} e_i \odot e_i\}_{1\leq i\leq n}$ is an orthonormal basis for $S^2(V)$.

The space of symmetric two-tensors on $\Lambda^2(V)$ has the orthogonal decomposition 
\begin{equation*}
    S^2(\Lambda^2 (V))=S^2_B(\Lambda^2 (V)) \oplus \Lambda^4 (V),
\end{equation*}
where $S^2_B(\Lambda^2 (V))$ consists of all tensors $R\in S^2(\Lambda^2 (V))$ that also satisfy the first Bianchi identity. 
Any $R\in S^2_B(\Lambda^2 (V))$ is called an algebraic curvature operator. 

By the symmetries of $R\in  S^2_B(\Lambda^2 (V))$ (not including the first Bianchi identity), there are (up to sign) two ways that $R$ can induce a symmetric linear map $R:T^2(V) \to T^2(V)$. 
The first one, denoted by $\hat{R}: \Lambda^2(V) \to \Lambda^2(V)$ in this paper, is the so-called curvature operator defined by
\begin{equation}\label{eq R hat}
    \hat{R}(e_i\wedge e_j) =\frac 1 2 \sum_{k,l}R_{ijkl}e_k \wedge e_l,
\end{equation}
where $\{e_1, \cdots, e_n\}$ is an orthonormal basis of $V$. 
Note that if the eigenvalues of $\hat{R}$ are all greater than or equal to $\kappa \in \R$, then all the sectional curvatures of $R$ are bounded from below by $\kappa$.

The second one, denoted by $\mathring{R}:S^2(V) \to S^2(V)$, is defined by 
\begin{equation}\label{eq R ring}
    \mathring{R}(e_i \odot e_j) =\sum_{k,l}R_{iklj} e_k \odot e_l.
\end{equation}
However, on contrary to the case of $\hat{R}$, all eigenvalues of $\mathring{R}$ being nonnegative implies all the sectional curvatures of $R$ are zero, that it, $R\equiv 0$ \footnote{This follows from the observation that the trace of $\mathring{R}:S^2(V) \to S^2(V)$ is equal to $\frac{S}{2}$ and $\mathring{R}(g,g)=-S$, where $S$ denotes the scalar curvature.}. 
The new feature here is that $S^2(V)$ is not irreducible under the action of the orthogonal group $O(V)$ of $V$. 
The space $S^2(V)$ splits into $O(V)$-irreducible subspaces as 
\begin{equation*}
    S^2(V)=S^2_0(V) \oplus \R g,
\end{equation*}
where $S^2_0(V)$ denotes the space of traceless symmetric two-tensors on $V$.
The map $\mathring{R}$ defined in \eqref{eq R ring} then induces a symmetric bilinear form
\begin{equation}\label{eq 2.3}
    \mathring{R}:S^2_0(V) \times S^2_0(V) \to \R
\end{equation}
by restriction to $S^2_0(V)$. 
Note that if all the eigenvalues of $\mathring{R}$ restricted to $S^2_0(V)$ are bounded from by $\kappa \in \R$, then the sectional curvatures of $R$ are bounded from below by $\kappa$. 
It should be noted that $\mathring{R}$ does not preserve the subspace $S^2_0(V)$ in general, but it does, for instance, when $R$ is Einstein.

Following \cite{Nishikawa86}, we call the symmetric bilinear form $\mathring{R}$ in \eqref{eq 2.3}
the \textit{curvature operator of the second kind}\footnote{It was pointed out in \cite{NPW22} that the curvature operator of the second kind can also be interpreted as the self-adjoint operator $\pi \circ \mathring{R}:S^2_0(V) \to S^2_0(V)$ with $\pi$ being the projection from $S^2(V)$ onto $S^2_0(V)$. This is equivalent to its interpretation as the symmetric bilinear form in \eqref{eq 2.3}, as 
$  
\mathring{R}(\vp,\psi)=\langle \mathring{R}(\vp), \psi\rangle =\langle \pi \circ \mathring{R} (\vp), \psi \rangle=  (\pi \circ \mathring{R})(\vp,\psi)
$
for any $\vp, \psi \in S^2_0(V)$.}, to distinguish it from the map $\hat{R}$ defined in \eqref{eq R hat}, which he called the \textit{curvature operator of the first kind}.

The action of the Riemann curvature tensor on symmetric two-tensors indeed has a long history. It appeared for K\"ahler manifolds in the study of deformation of complex analytic structures by Calabi and Vesentini \cite{CV60}, who introduced the self-adjoint operator $\xi_{\a \b} \to R^{\rho}_{\ \a\b}{}^{\sigma} \xi_{\rho \sigma}$ from $S^2(T^{1,0}_p M)$ to itself, and computed the eigenvalues of this operator on Hermitian symmetric spaces of classical type, with the exceptional ones handled shortly afterward by Borel \cite{Borel60}. In the Riemannian setting, the operator $\mathring{R}$ arises naturally in the context of deformations of Einstein structure in Berger and Ebin \cite{BE69} (see also \cite{Koiso79a, Koiso79b} and \cite{Besse08}). 
In addition, it appears naturally in the Bochner-Weitzenb\"ock formulas for symmetric two-tensors (see for exampe \cite{MRS20}), for differential forms in \cite{OT79} and for Riemannian curvature tensors in \cite{Kashiwada93}.  
In another direction, curvature pinching estimates for $\mathring{R}$ was studied by Bourguignon and Karcher \cite{BK78}, and they also calculated eigenvalues of $\mathring{R}$ on the complex projective space with the Fubini-Study metric and the quaternionic projective space with its canonical metric. 
Nevertheless, the curvature operator of the second kind is significantly less investigated than the curvature operator of the first kind. 

Let $N=\dim(S^2_0(V))=\frac{(n-1)(n+2)}{2}$ and $\{\vp_i\}_{i=1}^N$ be an orthonormal basis of $S^2_0(V)$. The $N\times N$ matrix $\mathring{R}(\vp_i, \vp_j)$ is called the matrix representation of $\mathring{R}$ with respect to the orthonormal basis $\{\vp_i\}_{i=1}^N$. 
The eigenvalues of $\mathring{R}$ refers to the eigenvalues of any of its matrix representation. 
This is independent of the choices of the orthonormal bases because matrix representations of $\mathring{R}$ with respect to different orthonormal bases of $S^2_0(V)$ are similar to each other. 

For a positive integer $1\leq k \leq N$, we say $R\in S^2_B(\Lambda^2(V))$ has $k$-nonnegative curvature operator of the second kind if the sum of the smallest $k$-eigenvalues of $\mathring{R}$ is nonnegative. 
This definition was extended to all $k\in [1,N]$ in \cite{Li22JGA} as follows. 
\begin{definition}
Let $N=\frac{(n-1)(n+2)}{2}$ and $\a \in [1, N]$. 
\begin{enumerate}
    \item We say $R\in S^2_B(\Lambda^2(V))$ has $\a$-nonnegative curvature operator of the second kind if for any orthonormal basis $\{\vp_i\}_{i=1}^{N}$ of $S^2_0(V)$, it holds that 
\begin{equation*}\label{eq def R}
    \sum_{i=1}^{\lfloor \a \rfloor} \mathring{R}(\vp_i,\vp_i) +(\a -\lfloor \a \rfloor) \mathring{R}(\vp_{\lfloor \a \rfloor+1},\vp_{\lfloor \a \rfloor+1}) \geq  0. 
\end{equation*}
If the inequality is strict then $R$ is said to have $\a$-positive curvature operator of the second kind.
 \item We say $R\in S^2_B(\Lambda^2(V))$ has $\a$-nonpositive (respectively, $\a$-negative) curvature operator of the second kind if $-R$ has $\a$-nonnegative (respectively, $\a$-positive) curvature operator of the second kind.
\end{enumerate} 
\end{definition}

\begin{definition}
A Riemannian manifold $(M^n,g)$ is said to have $\a$-nonnegative (respectively, $\a$-positive, $\a$-nonpositive, $\a$-negative) curvature operator of the second kind if $R_p \in S^2_B(\Lambda^2 T_pM)$ has $\a$-nonnegative (respectively, $\a$-positive, $\a$-nonpositive, $\a$-negative) curvature operator of the second kind for each $p\in M$.    
\end{definition}

\section{Six-positivity}

In this section, we prove Theorem \ref{thm 6PCO Riem}. 
\begin{proof}[Proof of Theorem \ref{thm 6PCO Riem}]
The key idea is to apply $\mathring{R}$ to the eigenvectors of the curvature operator of the second kind on the model space $\mathbb{S}^2 \times \mathbb{S}^2$. 

Fix $p\in M$ and let $V=T_pM$. 
Given an orthonormal four-frame $\{e_1, e_2,e_3,e_4\}$ in $V$, we define the following traceless symmetric two-tensors on $V$:
\begin{equation*}
 \vp_1 = \frac 1 4 \left( e_1\odot e_1+e_2\odot e_2 -e_3\odot e_3 -e_4 \odot e_4 \right),    
\end{equation*}
\begin{eqnarray*}
\vp_2 &=& \frac{1}{\sqrt{2}}e_1\odot e_3, \\
\vp_3 &=& \frac{1}{\sqrt{2}}e_1\odot e_4, \\
\vp_4 &=& \frac{1}{\sqrt{2}}e_2 \odot e_3, \\
\vp_5 &=& \frac{1}{\sqrt{2}}e_2 \odot e_4, \\
\vp_6 &=& \frac{1}{\sqrt{2}}e_1 \odot e_2, \\
\vp_7 &=& \frac{1}{\sqrt{2}}e_3 \odot e_4, 
\end{eqnarray*}
and
\begin{eqnarray*}
\vp_8 &=& \frac{1}{2\sqrt{2}}\left( e_1 \odot e_1 -e_2\odot e_2 \right), \\
\vp_9 &=& \frac{1}{2\sqrt{2}} \left( e_3 \odot e_3 -e_4\odot e_4 \right). 
\end{eqnarray*}
One easily verifies that $\{\vp_1, \cdots, \vp_9\}$ form an orthonormal subset in $S^2_0(V)$.

Suppose that $M$ has six-nonnegative curvature operator of the second kind, then we have
\begin{equation*}
    \sum_{i=1}^5\mathring{R}(\vp_i,\vp_i)+\mathring{R}(\vp_k,\vp_k) \geq 0
\end{equation*}
for $k=6, 7, 8, 9$.
Averaging the above four inequalities yields
\begin{equation}\label{eq key}
    \sum_{i=1}^5\mathring{R}(\vp_i,\vp_i) + \frac{1}{4} \left(\mathring{R}(\vp_6,\vp_6)+\mathring{R}(\vp_7,\vp_7)  + \mathring{R}(\vp_8,\vp_8) + \mathring{R}(\vp_9,\vp_9) \right) \geq 0. 
\end{equation}

Next, we calculate $\mathring{R}(\vp_i,\vp_i)$ for $1\leq i \leq 9$. 
Noticing that the only non-vanishing terms of $\vp_1$ are 
\begin{equation*}
    (\vp_1)_{11}=(\vp_1)_{11}=\frac 1 2 \text{ and }  (\vp_1)_{33}=(\vp_1)_{44}=-\frac 1 2,
\end{equation*}
we compute that 
\begin{eqnarray*}
\mathring{R}(\vp_1,\vp_1) &=&  \sum_{i,j,k,l=1}^n R_{ijkl}(\vp_1)_{il}(\vp_1)_{jk} \\
&=& \sum_{i,j=1}^n R_{ijji}(\vp_1)_{ii}(\vp_1)_{jj} \\
&=& \frac 1 2 \left(-R_{1212}-R_{3434}+R_{1313}+R_{2424}+R_{1414}+R_{2323} \right)
\end{eqnarray*}
Similarly, direct calculation shows that
\begin{eqnarray*}
\mathring{R}(\vp_2,\vp_2) &=& R_{1313} , \\
\mathring{R}(\vp_3,\vp_3) &=& R_{1414} , \\
\mathring{R}(\vp_4,\vp_4) &=& R_{2323} , \\
\mathring{R}(\vp_5,\vp_5) &=& R_{2424}, \\
\mathring{R}(\vp_6,\vp_6) &=& R_{1212},  \\
\mathring{R}(\vp_7,\vp_7) &=& R_{3434}, \\
\mathring{R}(\vp_8,\vp_8) &=& R_{1212},  \\
\mathring{R}(\vp_9,\vp_9) &=& R_{3434}. 
\end{eqnarray*}
Plugging the above nine identities into \eqref{eq key} produces, after simplification, that 
\begin{equation*}
    \frac{3}{2} \left(R_{1313}+R_{1414}+R_{2323}+R_{2424}  \right) \geq 0.
\end{equation*}

Similarly, if $M$ has six-positive curvature operator of the second kind, then \eqref{eq key} becomes strict and we conclude that 
\begin{equation*}
    R_{1313}+R_{1414}+R_{2323}+R_{2424}  > 0.
\end{equation*}
The proof is complete.
\end{proof}

\section{K\"ahler manifolds}

In the section, we present the proofs of Theorems \ref{thm 6PCO kahler} and \ref{thm main}. 

\begin{proof}[Proof of Theorem \ref{thm 6PCO kahler}]
Recall that on a K\"ahler manifold there exists a section $J$ of the endomorphism bundle $\text{End}(TM)$ with following properties: 
\begin{enumerate}
    \item $J$ is parallel;
    \item for each point $p\in M$, we have $J^2=-\id$ and $g(X,Y)=g(JX,JY)$ for all $X,Y\in T_pM$;
    \item the Riemann curvature tensor satisfies 
    $$R(X,Y,Z,W)=R(X,Y,JZ,JW)$$ for all $X,Y,Z,W\in T_pM$.
\end{enumerate} 

Given two unit vectors $X,Y \in T_pM$ satisfying $g(X,Y)=g(X,JY)=0$, we 
have by (2) that
\begin{equation*}
    g(JX, Y)=g(J^2X,JY)=-g(X,JY) =0
\end{equation*} 
and $$g(JX,JY)=g(X,Y)=0.$$ 
Thus, $\{X,JX,Y,JY\}$ is an orthonormal four-frame if $g(X,Y)=g(X,JY)=0$.

Since $M$ has six-positive curvature operator of the second kind, we have by Theorem \ref{thm 6PCO Riem} that
\begin{align*}
   & R(X,Y,X,Y)+R(JX,Y,JX,Y)\\
 &+ R(X,JY,X,JY)+R(JX,JY,JX,JY) > 0.
\end{align*}
On the other hand, it follows from the first Bianchi identity and (3) that 
\begin{eqnarray*}
 2 R(X,JX,Y,JY)& = & R(X,Y,X,Y)+R(JX,Y,JX,Y)\\
 &&+ R(X,JY,X,JY)+R(JX,JY,JX,JY).
\end{eqnarray*}
Thus, we conclude that 
\begin{equation}\label{eq 4.3}
    R(X,JX,Y,JY)>0
\end{equation}
for all unit vectors $X,Y \in T_pM$ satisfying $g(X,Y)=g(X,JY)=0$.
In other words, $M$ has positive orthogonal bisectional curvature. 

Similarly, six-nonnegative curvature operator of the second kind implies nonnegative orthogonal bisectional curvature for K\"ahler manifolds.

\end{proof}

\begin{proof}[Proof of Theorem \ref{thm main}]
By Theorem \ref{thm 6PCO kahler}, if $M$ has six-positive curvature operator of the second kind, then $M$ has positive orthogonal bisectional curvature. Theorem \ref{thm main} follows from the classification of closed K\"ahler manifolds with positive orthogonal bisectional curvature (see for example \cite[Corollary 3.2]{GZ10} or \cite{Wilking13}). 

\end{proof}

\section{Rigidity of $\mathbb{S}^2 \times \mathbb{S}^2$}

We prove Theorem \ref{thm 6NCO} in this section. We first prove the following lemma.

\begin{lemma}\label{lemma 2 by 2}
Suppose the Riemannian product manifold $(M,g)=(M_1,g_1) \times (M_2,g_2)$ with $\dim(M_1)=\dim(M_2)=2$ has six-nonnegative curvature operator of the second kind. Then either $M$ is flat or the universal cover of $M$ is isometric to $\mathbb{S}^2 \times \mathbb{S}^2$ with the same round metric on both factors. 
\end{lemma}
\begin{proof}
Pick $p\in M_1$, $q\in M_2$ and then choose an orthonormal basis $\{e_1, e_2\}$ of $T_pM_1$ and an orthonormal basis $\{e_3, e_4\}$ of $T_q M_2$. 
Then $\{e_1, e_2, e_3, e_4\}$ is an orthonormal basis of $T_{p,q}M$.
Let $\{\vp_1, \cdots, \vp_9\}$ be the basis of $S^2_0(T_{p,q}M)$ defined the same as in Section 3. Then we have 
\begin{equation*}
    \mathring{R}(\vp_1,\vp_1)=-\frac 1 2 \left(R_{1212}+R_{3434} \right),
\end{equation*}
\begin{equation*}
    \mathring{R}(\vp_i,\vp_i)=0 \text{ for } i=2,3,4,5, 
\end{equation*}
and 
\begin{eqnarray*}
    \mathring{R}(\vp_6,\vp_6)=\mathring{R}(\vp_8,\vp_8)=R_{1212};\\
    \mathring{R}(\vp_7,\vp_7)=\mathring{R}(\vp_9,\vp_9)=R_{3434}.\\
\end{eqnarray*}
Since $M$ has six-nonnegative curvature operator of the second kind, we get 
\begin{eqnarray*}
2R_{1212} &=& \sum_{i=2}^5 \mathring{R}(\vp_i,\vp_i) +\mathring{R}(\vp_6,\vp_6)+\mathring{R}(\vp_8,\vp_8) \geq 0; \\
2R_{3434} &=& \sum_{i=2}^5 \mathring{R}(\vp_i,\vp_i) +\mathring{R}(\vp_7,\vp_7)+\mathring{R}(\vp_9,\vp_9) \geq 0; \\
\end{eqnarray*}
and 
\begin{eqnarray*}
\frac 1 2 \left(R_{3434}-R_{1212} \right) &=& \sum_{i=2}^5 \mathring{R}(\vp_i,\vp_i) +\mathring{R}(\vp_6,\vp_6)+\mathring{R}(\vp_1,\vp_1) \geq 0; \\
\frac 1 2 \left(R_{1212}-R_{3434} \right) &=& \sum_{i=2}^5 \mathring{R}(\vp_i,\vp_i) +\mathring{R}(\vp_7,\vp_7)+\mathring{R}(\vp_1,\vp_1) \geq 0. \\
\end{eqnarray*}
Therefore, we must have 
$$R_{1212}=R_{3434}\geq 0.$$ 
It follows that both $M_1$ and $M_2$ have nonnegative constant scalar curvature, so $M$ is either flat or the universal cover of $M$ isometric to $\mathbb{S}^2 \times \mathbb{S}^2$, where both factors are equipped with the same round metric.
\end{proof}

\begin{proof}[Proof of Theorem \ref{thm 6NCO}]
Since $M$ has six-nonnegative curvature operator of the second kind, it has nonnegative orthogonal bisectional curvature by Theorem \ref{thm 6PCO kahler}. 

Let $\widetilde{M}$ be the universal cover of $M$.
According to \cite[Theorem 1.3]{GZ10}, if $M$ is non-flat, then $\widetilde{M}$ is either biholomorphic to $\mathbb{CP}^2$ or splits isometrically as $M_1 \times M_2$ with $\dim(M_1)=\dim(M_2)=2$. 
In view of Lemma \ref{lemma 2 by 2},  
both $M_1$ and $M_2$ are isometric to $\mathbb{S}^2$ with the same round metric.
Hence $M$ is either biholomorphic to $\mathbb{CP}^2$ or isometric to $\mathbb{S}^2 \times \mathbb{S}^2$. 
\end{proof}

\section{Rigidity of $\mathbb{CP}^2$}

Let $(M^4,g)$ be an oriented four-manifold. 
The Hodge star operator $ *: \Lambda^2 \to \Lambda^2$, where $\Lambda^2$ is the bundle of two-forms, induces a splitting 
\begin{equation*}
    \Lambda^2 = \Lambda^+ \oplus \Lambda^1, 
\end{equation*}
where $\Lambda^\pm$ is the $\pm1$-eigenspace of $*$. 
As a consequence, the curvature operator $\hat{R}: \Lambda^2 \to \Lambda^2$ decomposes as 
\begin{equation*}
    \begin{pmatrix}
        \frac{S}{12} \id +W^+ & \overset{\circ}{\Ric} \\
        \overset{\circ}{\Ric} & \frac{S}{12} \id +W^-,
    \end{pmatrix}
\end{equation*}
where $W^\pm$ are the restriction of the Weyl curvature tensor to $\Lambda^\pm$.

Analogous to this decomposition for $\hat{R}$, Cao, Gursky and Tran \cite{CGT21} proved the following block decomposition for the matrix associated to $\mathring{R}$.
\begin{theorem}\label{thm CGT basis}
    Let $(M^4,g)$ be an oriented four-manifold. Then there exists a basis of $S^2_0(TM)$ with respect to which the matrix associated to $\mathring{R}$ is given by 
\begin{equation}\label{eq 6.0}
    \mathring{R}=
\begin{pmatrix}
    D_1 &  O_1 & O_2 \\
    -O_1 & D_2 & O_3 \\
    -O_2 & -O_3 & D_3
\end{pmatrix},
\end{equation}
and the $D_i$'s are diagonal matrices given by 
\begin{equation}\label{eq 6.01}
    D_i=
\begin{pmatrix}
    -(\l_i+\mu_1)+\frac{1}{12}S & & \\
    & -(\l_i+\mu_2)+\frac{1}{12}S & \\
    & & -(\l_i+\mu_3)+\frac{1}{12}S
\end{pmatrix},
\end{equation}
where $\{\l_1,\l_2, \l_3\}$ are the eigenvalues of $W^+$ and $\{\mu_1,\mu_2,\mu_3\}$ are the eigenvalues of $W^-$. Moreover, $O_1,O_2,O_3$ are skew-symmetric $3\times 3$ matrices which vanish if and only if $(M^4,g)$ is Einstein.
\end{theorem}

We shall use the above normal form of $\mathring{R}$ to prove Theorem \ref{thm Kahler 2D}. 

\begin{proof}[Proof of Theorem \ref{thm Kahler 2D}]
It is well-known (see for example \cite{Derdzinski83}) that on any K\"ahler surface $(M^4,g,J)$ with the natural orientation (in the sense that the K\"ahler form $\omega \in \Lambda^+$), the self-dual Weyl operator $W^+:\Lambda^2 \to \Lambda^2$ is given by 
\begin{equation*}
    W^+=\begin{pmatrix}
        -\frac{S}{12} & & \\
        & -\frac{S}{12} & \\
        & & \frac{S}{6}
    \end{pmatrix}.
\end{equation*}

Fix $p\in M$ and denote by $\l_1=\frac{S}{6}, \l_2=\l_3=-\frac{S}{12}$ the eigenvalues of $W^+$ at $p$ and $\mu_1 \geq \mu_2 \geq \mu_3$ the eigenvalues of $W^-$ at $p$. 
Let $\{\vp_i\}_{i=1}^9$ be the basis of $S^2_0(T_pM)$ constructed in \cite{CGT21} such that the matrix associated to $\mathring{R}$ with respect to this basis is given by \eqref{eq 6.0}.

Using $\l_1=\frac{S}{6}, \l_2=\l_3=-\frac{S}{12}$, we have
\begin{equation}\label{eq 6.3}
    D_1=
\begin{pmatrix}
    -\mu_1-\frac{1}{12}S & & \\
    & -\mu_2-\frac{1}{12}S & \\
    & & -\mu_3-\frac{1}{12}S
\end{pmatrix},
\end{equation}
and
\begin{equation}\label{eq 6.4}
    D_2=D_3=
\begin{pmatrix}
    -\mu_1+\frac{1}{6}S & & \\
    & -\mu_2+\frac{1}{6}S & \\
    & & -\mu_3+\frac{1}{6}S
\end{pmatrix}.
\end{equation}

Since $\mathring{R}$ is $4\frac{1}{2}$-nonnegative, we have that 
\begin{eqnarray*}
     0 & \leq & \sum_{i=1}^3 \mathring{R}(\vp_i,\vp_i) +\mathring{R}(\vp_4,\vp_4) +\frac{1}{2} \mathring{R}(\vp_7,\vp_7). 
\end{eqnarray*}
Using $\mu_1+\mu_2+\mu_3=0$ and the fact that $\mathring{R}$ is represented by the matrix in \eqref{eq 6.0}, the above inequality becomes
\begin{eqnarray*}
    0    &\leq & \left(-\mu_1 -\mu_2 -\mu_3 -\frac{S}{4}\right) +\frac{3}{2}\left(-\mu_1 +\frac{S}{6} \right) \\
    &=& -\frac 3 2 \mu_1,
\end{eqnarray*}
Therefore we have $\mu_1 \leq 0$. 
Thus we must have $\mu_1=\mu_2=\mu_3=0$, i.e., $W^-\equiv 0$.

Next, we show that $O_1=O_2=O_3=0$, i.e., $M$ is Einstein. 
To see $(O_3)_{12}=\mathring{R}(\vp_4,\vp_8)=0$, we consider the orthonormal subset $\{\psi_i\}_{i=1}^5$ of $S^2_0(T_pM)$ with $\psi_i=\vp_i$ for $i=1,2,3,5$, $\psi_4=\frac{1}{\sqrt{2}}(\vp_4 \pm \vp_8)$. Since $M$ has $4\frac{1}{2}$-nonnegative curvature operator of the second kind, we have 
\begin{eqnarray*}
    0 &\leq& \sum_{i=1,2,3,5}\mathring{R}(\psi_i,\psi_i) +\frac{1}{2} \mathring{R}(\psi_4,\psi_4) \\
    &=& \sum_{i=1,2,3,5}\mathring{R}(\vp_i,\vp_i) + \frac{1}{4}\left(\mathring{R}(\vp_4,\vp_4)+\mathring{R}(\vp_8,\vp_8)\pm 2\mathring{R}(\vp_4,\vp_8) \right) \\
    &=& \pm 2\mathring{R}(\vp_4,\vp_8),
\end{eqnarray*}
where we have used $\mathring{R}(\vp_i,\vp_i)=-\frac{S}{12}$ for $i=1,2,3$ and $\mathring{R}(\vp_i,\vp_i)=\frac{S}{6}$ for $i\neq 1,2,3$. 
Therefore, we have $(O_3)_{12}=\mathring{R}(\vp_4,\vp_8)=0$. Similar arguments yield $(O_3)_{13}=(O_3)_{23}=0$. Hence, $O_3=0$. 

To prove $(O_1)_{12}=\mathring{R}(\vp_1,\vp_5)=0$, let's consider the orthonormal subset $\{\psi_i\}_{i=1}^5$ of $S^2_0(T_pM)$ with $\psi_1=(\cos t) \vp_1 +(\sin t) \vp_5$, $\psi_i=\vp_i$ for $i=2,3,4$, $\psi_5=\vp_6$. Define a function
\begin{eqnarray*}
    f(t) &:=& \sum_{i=1,2,3,4}\mathring{R}(\psi_i,\psi_i) + \frac{1}{2} \mathring{R} (\psi_5,\psi_5) \\
    &=& \cos^2 t \mathring{R} (\vp_1,\vp_1) + \sin^2 t \mathring{R} (\vp_5,\vp_5) + 2 \sin t \cos t \mathring{R}(\vp_1,\vp_5) \\ 
    && + \sum_{i=2,3,4}\mathring{R}(\vp_i,\vp_i)+\frac{1}{2} \mathring{R} (\vp_6,\vp_6) \\
    &=& (1-\cos^2 t) \frac{S}{12} +\sin^2 t \frac{S}{6} + 2 \sin t \cos t  (O_1)_{12}.
\end{eqnarray*}
Since $f(t) \geq 0$ and $f(0)=0$, we conclude that $f'(0)=2(O_1)_{12}=0$.
Similarly, one can show that $(O_1)_{13}=(O_1)_{23}=0$, and hence $O_1=0$. Using the same idea, one also gets $O_2=0$. 

Thus, we have proved that the manifold is K\"ahler-Einstein and half-locally conformally flat ($W^-\equiv 0$). Note that we have $S\geq 0$. By a result of Hitchin \cite{Hitchin74}, $M$ is either flat or isometric to $\mathbb{CP}^2$ with the Fubini-Study metric, up to scaling.
\end{proof}

\begin{remark}
The above proof is poinwise and algebraic. By flipping the signs, one can show that a K\"ahler surface with $4\frac{1}{2}$-nonpositive curvature operator of the second kind has constant nonpositive holomorphic sectional curvature\footnote{This fact has also been proved by the author in \cite{Li22Kahler} using a different method.}.
\end{remark}

\begin{remark}
One can use the normal form of $\mathring{R}$ in Theorem \ref{thm CGT basis} to give an alternative proof the fact that a K\"ahler surface with six-nonnegative curvature operator of the second kind has nonnegative orthogonal bisectional curvature. The argument goes as follows. 
If $M$ has six-nonnegative curvature operator of the second kind, then the sum of any six diagonal elements of the matrix in \eqref{eq 6.0} of $\mathring{R}$ is nonnegative. This implies that the matrix 
\begin{equation*}
    A:=\begin{pmatrix}
    D_1 &  0 & 0 \\
    0 & D_2 & 0\\
    0 & 0 & D_3
\end{pmatrix}
\end{equation*}
is also six-nonnegative. Using the K\"ahlerity, we have that $D_1$ is given by \eqref{eq 6.3}, and $D_2$ and $D_3$ are given by \eqref{eq 6.4}. It's not hard to see that six-nonnegativity of $A$ implies $\mu_3 \geq -\frac{1}{12}S$, where $\mu_1 \geq \mu_2 \geq \mu_3$ are the eigenvalue of $W^{-}$. 
So $M$ has nonnegative isotropic curvature, which is equivalent to $\mu_2+\mu_3 \geq -\frac{1}{6}S$ for K\"ahler surfaces (see \cite{MW93} or \cite{Hamilton97}). 
The statement follows as nonnegative isotropic curvature implies nonnegative orthogonal bisectional curvature for K\"ahler manifold.

\end{remark}



\bibliographystyle{alpha}
\bibliography{ref}

\end{document}